\documentclass[11pt]{article}
\usepackage{amsmath, amsthm, amssymb}
\usepackage{tkz-graph, subfig}
\usepackage{marginnote}
\usepackage{multicol}
\usepackage{verbatim}
\usepackage{fullpage}

\usepackage{color, textcomp}
\usepackage[font=scriptsize,margin=.7in]{caption}

\usepackage[]{hyperref}

\theoremstyle{plain}
\newtheorem{thm}{Theorem}

\newtheorem{obs}{Observation}

\newtheorem*{mainthm}{Main Theorem}
\newtheorem*{kolm}{Kolmogorov's 0--1 Law}

\theoremstyle{definition}

\theoremstyle{remark}

\newcommand{\fancy}[1]{\mathcal{#1}}

\def\A{\fancy{A}}

\def\D{\fancy{D}}

\newcommand{\I}{\mathcal{I}}

\newcommand{\T}{\mathcal{T}}

 \newcommand{\aside}[1]{\marginnote{\scriptsize{#1}}[0cm]}
 \newcommand{\aaside}[2]{\marginnote{\scriptsize{#1}}[#2]}
\newcommand\Emph[1]{\emph{#1}\aside{#1}}
\newcommand\EmphE[2]{\emph{#1}\aaside{#1}{#2}}
\newcommand\Memph[1]{#1\aside{#1}}

\title{A Note on Bootstrap Percolation Thresholds\\ 
in Plane Tilings using Regular Polygons}
\author{
Neal Bushaw\thanks{Department of Mathematics and Applied 
Mathematics, Virginia Commonwealth University, Richmond, VA, U.S.A.;
\texttt{nobushaw@vcu.edu}}
\and
Daniel W. Cranston\thanks{Department of Mathematics and Applied
Mathematics, Virginia Commonwealth University, Richmond, VA, U.S.A.;
\texttt{dcranston@vcu.edu}; 
This research is partially supported by NSA Grant 
H98230-16-0351.}
}

\begin{document}
\maketitle
\abstract{
In \emph{$k$-bootstrap percolation}, we fix $p\in (0,1)$, an integer $k$, and a
plane graph $G$.  Initially, we infect each face of $G$ independently with
probability $p$.  Infected faces remain infected forever, and if a healthy
(uninfected) face has at least $k$ infected neighbors, then it becomes infected.
For fixed $G$ and $p$, the \emph{percolation threshold} is the largest $k$ such
that eventually all faces become infected, with probability at least $1/2$.
For many infinite graphs, we show that this threshold is independent of $p$.

We consider bootstrap percolation in tilings of the plane by regular polygons.
A \Emph{vertex type} in such a tiling is the cyclic order of the faces that
meet a common vertex.  First, we determine the percolation threshold for each
of the Archimedean lattices.  
More generally, let $\mathcal{T}$ denote the set of plane tilings $T$ by
regular polygons such that if $T$ contains one instance of a vertex type, then
$T$ contains infinitely many instances of that type.  We show that no tiling in
$\mathcal{T}$ has threshold 4 or more.  Further, the only tilings
in $\mathcal{T}$ with threshold 3 are four of the Archimedean lattices.
Finally, we describe a large subclass of $\mathcal{T}$ with threshold 2.
}

\section{Introduction}
In \emph{$k$-bootstrap percolation}, we fix $p\in (0,1)$, an integer $k$, and a
plane graph $G$.  Initially, we infect each face of $G$ independently with
probability $p$; call the set of initially infected faces $\I$.  Infected faces
remain infected forever, and if a healthy (uninfected) face has at least $k$
infected neighbors, then it becomes infected.  We say that \Emph{$\I$
percolates} if eventually all faces become infected.  For short, we call this
the \Emph{$k$-bootstrap model}.  For fixed $G$ and
$p$, the \emph{percolation threshold}\footnote{Note that this is different than
the \emph{probability} thresholds often considered for sequences of finite
graphs.}, or simply \EmphE{threshold}{3mm}, is the
largest $k$ such that in the $k$-bootstrap model $\I$ percolates with
probability at least $1/2$.
For a large class of infinite graphs, we show that the
threshold is independent of $p$.  

The $k$-bootstrap model has a long, rich history.  Introduced by
Chalupa, Leath, and Reich \cite{ChalupaLR} in 1979 as a way to model magnetic
materials, it is an example of a monotone cellular automata (introduced by von
Neumann \cite{vonneumann} in 1966).  Most of the work in this field has focused
on finding thresholds for growing families of graphs.  For example, if we
infect each face of the $n\times n$ square grid independently with some
probability $p$, how large must $p$ be so the infection percolates almost
surely, as $n\to\infty$?  The answer to this question, and the first sharp
result in the area, was proved by Holroyd \cite{Holroyd}.  While Holroyd's
result is striking on its own, it has been extended greatly: studying the
problem in higher dimensions, finding more terms of the critical probability
function, and much more (see, e.g.,~\cite{AizemanL, BaloghBCM, BaloghBM, CerfC,
CerfM}).  These bootstrap models have been generalized significantly in recent
years, with the advent of graph bootstrap percolation~\cite{BollobasSU}.

Outside the realm of grids, bootstrap percolation has been studied on many
different families of graphs.  This includes work determining
critical probabilities for random regular graphs \cite{BaloghP}, the
Erd\H{o}s-Renyi random graph $G_{n,p}$ \cite{HolmgrenJK, JansonLTV}, the
hypercube \cite{BaloghBM2}, infinite trees \cite{BaloghPP}, and others.  Largely
ignored, however, has been percolation on infinite lattices (aside from the
square lattice \cite{Schonmann, vanEnter}, discussed below).  We explore this
direction here.

The \Emph{length} of a face of a plane graph is its number of sides.
A \EmphE{configuration}{4mm} is a finite plane graph.  A configuration $H$
\EmphE{appears in $G$}{4mm} if there is a map from faces of $H$ to faces of $G$
that preserves both face length and the number of edges shared by every pair of
faces.
When $H$ appears in $G$, we also say that $G$ \Emph{contains a copy~of~$H$}.
The following observation is straightforward, but it is our main tool for
proving upper bounds on percolation thresholds.

\begin{obs}
\label{obs1}
Let $C$ be a configuration such that each face of $C$ has at most $k$
neighboring faces outside $C$.  If $G$ contains infinitely many copies of
$C$, then $G$ has percolation threshold at most $k$.
\end{obs}
\begin{proof}
Suppose we are in the $(k+1)$-bootstrap model.
Note that if some copy of $C$ has no initially infected face, then $\I$ does not
percolate, since no face in that copy of $C$ ever becomes infected.  Since
$G$ has infinitely many copies of $C$ (and each face of $G$ is infected
independently), with probability 1 at least one copy of $C$ in $G$ has no face
initially infected.  So, in the $(k+1)$-bootstrap model, $\I$ percolates with
probability~0.
\end{proof}

An immediate consequence of Observation~\ref{obs1} is that the (infinite)
square lattice has percolation threshold at most 2, since we can take as our
configuration $C$ four square faces that meet at a common vertex.
van Enter~\cite{vanEnter} famously proved a matching lower
bound.  That is, the percolation threshold of the square lattice is 2.  In this
note, we extend this result, using the same approach, to determine the percolation thresholds for many
tilings of the plane by regular polygons.  Beyond this, we prove that, somewhat
surprisingly, for a large class of graphs (those whose vertex types repeat
infinitely often) the percolation threshold is never more than four, and the
only tilings achieving this are Archemidean lattices.  We also determine a
large class of tilings for which the threshold is exactly two.

\section{Archimedean Lattices}
A function $f:\mathbb{R}^2\to\mathbb{R}^2$ on a tiling is a \Emph{tiling
translation} if it has the form $f:(x,y)\mapsto(x+a,y+b)$ for some
$a,b\in\mathbb{R}$ and it maps the center of every $d$-gon to the center of a
$d$-gon.  These are simply translations of the plane which map our polygons to
congruent polygons.  As an example, for any $a,b\in\mathbb{Z}$, $f:(x,y)\mapsto
(x+a,y+b)$ is a tiling translation for the (unit) square lattice, but is not
a tiling
translation for the hex lattice when $a$ and $b$ are both nonzero, since the
height of a regular hexagon, with one side axis-aligned, is not a rational
multiple of its width.  An event $E$ is called \Emph{translation-invariant} if
for every initially infected set $\I$
and every tiling translation $f$ we have $f(\I)\in E$ if and only if $\I \in E$.  
For example, the event $E_1=\{\I: \I\textrm{ percolates to the entire
plane}\}$ is translation invariant; if a set $\I$ percolates, it will certainly
also percolate when that set is translated to another location in the plane,
since our percolation process is independent of a face's location in the plane.
 At the other extreme, as an example of a {\emph{non}}-translation-invariant
event consider $E_2=\{\I:\I\textrm{ infects the origin eventually}\}$.  
Let $\I_2=\{\textrm{only the face containing the origin is
infected}\}$.  Now $\I_2\in E$, but for any nontrivial tiling translation $f$ we
have $f(\I_2)\notin E_2$.  An event is \EmphE{weakly
translation-invariant}{-7mm} if
there exist infinitely many distinct tiling translations $f$ such that for each
initially infected set $\I$ we have $f(\I)\in E$ if and only if $\I\in E$.
Our main tool for proving lower bounds on percolation
thresholds is the following lemma of Kolmogorov about translation invariant
events.  This result is quite general, so we state it in a simple
form which is enough for our purposes.

\begin{kolm}
Let $T$ be an infinite graph that is locally finite. If $E$ is a
weakly translation-invariant event, then $\Pr(E)\in \{0,1\}$.
\end{kolm}

The proof is not hard, but requires enough machinery that we do not reproduce
it here.  Nevertheless, this lemma is
crucial to our work, so we give a brief description for the
probabilistically-minded reader. Our probability space is constructed
as a countable product space (whose fundamental events are whether or not
individual faces are infected).  As such, any event -- in particular, our
`initial set percolates' event -- can be approximated arbitrarily well by a
cylinder set (all hexagons within some fixed distance of a specified hexagon).  
Since we have weak translation-invariance, if we translate
sufficiently far we can find another approximating cylinder set which is
disjoint from the first; thus, events within one cylinder set are independent of
those within the other.
By repeating this process, we find infinitely many
disjoint copies of our approximating cylinder set. 
Each of these translations of our cylinder is initially entirely infected with
positive probability.  Hence, with probability one, at least one of these
approximating events will occur\footnote{The details are available, for example,
in \cite[p.~118ff.]{perc-book}.}.  

To show that the square lattice has threshold 2, Chalupa, Leath, and
Reich~\cite{ChalupaLR} defined an event $A$ with the
following three properties: (1) if $A$ occurs, then in the 2-bootstrap model
the initially infected set $\I$ percolates on the square lattice, (2) $A$
occurs with positive probability, and (3) $A$ is translation invariant. 
Properties (1) and (2) clearly imply that in the 2-bootstrap model on the
square lattice, $\I$ percolates with positive probability.  Now
Kolmogorov's 0--1 Law shows that this probability is 1.

An \Emph{Archimedean Lattice} is a vertex transitive (infinite) plane graph in
which each face is a regular polygon.
It is well-known that there are 11 such lattices\footnote{At the start of
Section~\ref{sec3}, we outline a proof of this fact.}, including the three
regular tilings (by the triangle, square, and hexagon).  To describe an
Archimedean Lattice, we write ($f_1.\ldots f_s$), where $f_1,\ldots,f_s$ are
the face lengths, in cyclic order, that meet at each vertex.  For instance, the
regular tilings by triangle, square, and hexagon are denoted (3.3.3.3.3.3),
(4.4.4.4), and (6.6.6).  In Figure~\ref{fig0} we show the other eight
Archimedean Lattices, along with configurations that bound their percolation
thresholds, via Observation~\ref{obs1}.  Clearly, every lattice has
percolation threshold at least 1.  For lattices (3.3.3.3.3.3), (3.3.3.3.3.6),
(3.3.3.4.4), (3.3.4.3.4), and (3.4.6.4) we obtain a matching upper bound,
using Observation~\ref{obs1} and the configurations in Figure~\ref{fig0}.  For
(3.6.3.6) our upper bound is 2, and for each of (3.12.12), (4.6.12), (4.8.8),
and (6.6.6) it
is 3.  So, to determine the bootstrap threshold for each of these five
lattices, the interesting work is proving a matching lower bound.  We first
present a proof for (4.8.8).  Since the proofs of all five lower bounds are
similar, we will just outline the differences for the remaining four lattices.

\captionsetup[subfigure]{labelformat=empty}
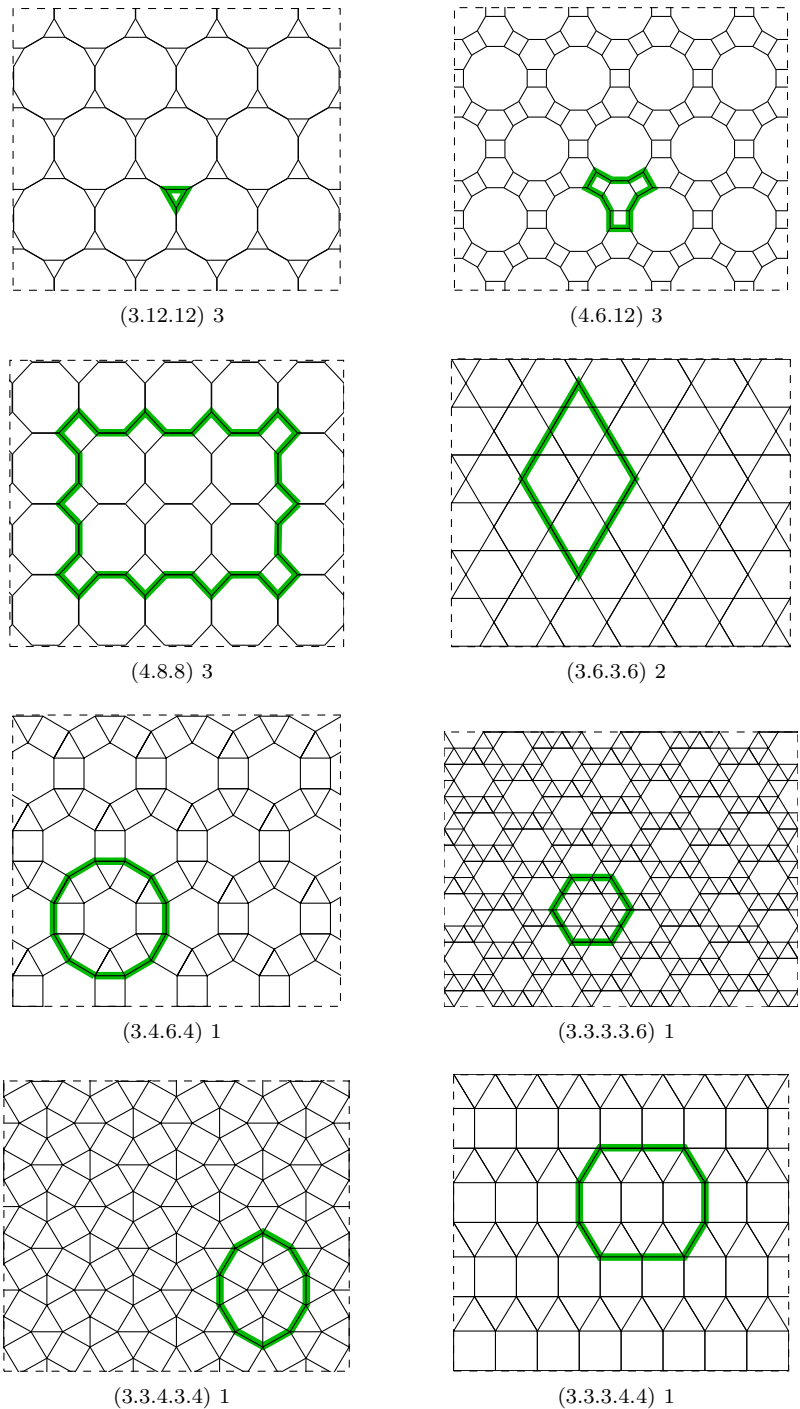
\begin{figure}[!b]
\centering
\subfloat[\scriptsize{(3.12.12)~3}]{
\makebox[.35\textwidth]{
\begin{tikzpicture}[scale = .5*.63*1.7]
\draw[green!73!black, line width = 1mm] (7.8,6) -- (8.35,6) -- (8.075,5.52) -- cycle;
\draw[dashed] (4.03,3.5) rectangle (12.15,10.5);
\begin{scope}
\clip (4.03,3.5) rectangle (12.15,10.5);
\foreach \y in {1,2,3}
\foreach \x in {1, ..., 6}
{
\begin{scope}[xshift=2.02*\x cm, yshift=3.51*\y cm]
\foreach \t in {15, 45, ..., 345}
{
\draw (\t:1.05) -- (\t+30:1.05);
}
\end{scope}
}
\begin{scope}[xshift=1.01 cm, yshift=1.7505cm]
\foreach \y in {1,2}
\foreach \x in {1, ..., 5}
{
\begin{scope}[xshift=2.02*\x cm, yshift=3.51*\y cm]
\foreach \t in {15, 45, ..., 345}
{
\draw (\t:1.05) -- (\t+30:1.05);
}
\end{scope}
}
\end{scope}
\end{scope}
\end{tikzpicture}

}}
\subfloat[\scriptsize{(4.6.12)~3}]{
\makebox[.35\textwidth]{
\begin{tikzpicture}[scale = .39*.63*1.7]
\draw[dashed] (5.3,4.5) rectangle (15.9,13.5);
\begin{scope}
\clip (5.3,4.5) rectangle (15.9,13.5);
\draw[green!73!black, line width = 1mm] 
(10.0,7.5) -- (9.503,7.76) -- (9.793,8.273) -- (10.3,8) -- (10.85,8) --
(11.33,8.273) -- (11.62,7.766) --
(11.125,7.5) -- (10.87,7.04) -- (10.87,6.47) -- (10.25,6.47) -- (10.25,7.04) -- 
cycle;

\foreach \y in {2,3}
\foreach \x in {1, ..., 7}
{
\begin{scope}[xshift=2.64*\x cm, yshift=4.5*\y cm]
\foreach \t in {15, 45, ..., 345}
\draw (\t:1.05) -- (\t+30:1.05);
%
\draw (1.01,.28) -- (1.63,.28);
\draw (1.01,-.28) -- (1.63,-.28);
\draw (.7475,-.7347) -- (1.050,-1.250);
\draw (.2625,-1.015) -- (.574, -1.53);
\draw (-.7475,-.7347) -- (-1.050,-1.250);
\draw (-.2625,-1.015) -- (-.574, -1.53);
\end{scope}
}

\begin{scope}[xshift=1.32 cm, yshift=2.25cm]
\foreach \y in {1,2}
\foreach \x in {1, ..., 6}
{
\begin{scope}[xshift=2.64*\x cm, yshift=4.5*\y cm]
\foreach \t in {15, 45, ..., 345}
\draw (\t:1.05) -- (\t+30:1.05);

\draw (1.01,.28) -- (1.63,.28);
\draw (1.01,-.28) -- (1.63,-.28);
\draw (.7475,-.7347) -- (1.050,-1.250);
\draw (.2625,-1.015) -- (.574, -1.53);
\draw (-.7475,-.7347) -- (-1.050,-1.250);
\draw (-.2625,-1.015) -- (-.574, -1.53);
\end{scope}
}
\end{scope}

\foreach \x in {1, ..., 7}
{
\begin{scope}[xshift=2.64*\x cm, yshift=4.5 cm]
\foreach \t in {15, 45, ..., 345}
\draw (\t:1.05) -- (\t+30:1.05);
%
\draw (1.01,.28) -- (1.63,.28);
\draw (1.01,-.28) -- (1.63,-.28);
\end{scope}
}
\end{scope}
\end{tikzpicture}
}}

\subfloat[\scriptsize{(4.8.8)~3}]{
\makebox[.35\textwidth]{
\begin{tikzpicture}[xscale=.45*.59*1.7, yscale=.45*.63*1.7]
\tikzstyle{unlabeledStyle}=[minimum size = 0pt, inner sep = 0pt, outer sep = 0pt]

\draw[green!73!black, line width=1mm] (8.20, 4.86) -- (8.8,5.44) -- (8.8,6.26) --
(8.22,6.84)-- (8.8,7.42) -- (9.37,6.85) -- (10.17,6.85) -- (10.76, 7.42) --
(11.35,6.85) -- (12.13, 6.85) -- (12.70,7.42) -- (13.27, 6.85) -- (14.10, 6.85)
-- (14.66,7.42) -- (15.22, 6.85) -- (14.66, 6.27) -- (14.68, 5.43) -- (15.26,
4.89);
\begin{scope}[yscale = -1, yshift=-9.78cm]
\draw[green!73!black, line width=1mm] (8.20, 4.86) -- (8.8,5.44) -- (8.8,6.26) --
(8.22,6.84)-- (8.8,7.42) -- (9.37,6.85) -- (10.17,6.85) -- (10.76, 7.42) --
(11.35,6.85) -- (12.13, 6.85) -- (12.70,7.42) -- (13.27, 6.85) -- (14.10, 6.85)
-- (14.66,7.42) -- (15.22, 6.85) -- (14.66, 6.27) -- (14.68, 5.43) -- (15.26,
4.89);
\end{scope}

\draw[dashed] (6.77,.95) rectangle (16.6,8.85);
\begin{scope}
\clip (6.77,.95) rectangle (16.6,8.85);
%
\foreach \y in {1, ..., 8}
{
\foreach \x in {1, ..., 8}
{
\begin{scope}[xshift=1.955*\x cm, yshift=1.955*\y cm]
\foreach \t in {22.5, 67.5, ..., 360}
{
\draw (\t:1.05) -- (\t+45:1.05);
}
\end{scope}
}
}
\end{scope}
\end{tikzpicture}
}}
\subfloat[\scriptsize{(3.6.3.6)~2}]{
\makebox[.35\textwidth]{
\centering
\begin{tikzpicture}[scale= .63*1.7]
\begin{scope}[xscale = .25*.7,  yscale = .433*.7*.98]
\draw[dashed] (.98,-.02) rectangle (25.02,12.02);
\draw[green!73!black, line width = 1mm] (10,3) -- (6,7) -- (10,11) -- (14,7) --
cycle;
\clip (1,0) rectangle (25,12);
\foreach \x in {0,...,5}
\foreach \y in {0,...,2}
{
\begin{scope}[xshift=4*\x cm, yshift=4*\y cm]
\draw (1,2) -- (0,1) -- (1,0) -- (3,0) -- (5,0) -- (4,1) -- (5,2) -- (6,3) --
(5,4) -- (3,4) -- (1,4) -- (2,3) -- (1,2) -- (3,2) -- (4,1) -- (3,0) (5,2) --
(3,2) -- (2,3) -- (3,4);
\end{scope}
}
\end{scope}
\end{tikzpicture}
}}

\subfloat[\scriptsize{(3.4.6.4)~1}]{
\makebox[.35\textwidth]{
\begin{tikzpicture}[xscale = .2*.60*.98*1.7, yscale=.2*.60*1.7]
\draw[dashed] (13.65,14.2) rectangle (35.450,33.2);
\begin{scope}
\clip (13.65,14.2) rectangle (35.45,33.2);
\draw[green!73!black, line width = 1mm] (19.1,16.2) 
-- ++ (2,0) -- ++ (1.732,1) -- ++ (1,1.732) 
-- ++ (0,2) -- ++ (-1,1.732) -- ++ (-1.732,1)
-- ++ (-2,0) -- ++ (-1.732,-1) -- ++ (-1,-1.732)
-- ++ (0,-2) -- ++ (1,-1.732) -- ++ (1.732,-1);
\foreach \y in {2,3}
\foreach \x in {2, ..., 6}
{
\begin{scope}[xshift=5.464*\x cm, yshift=9.464*\y cm]
\draw (0,2) -- (0,0) -- (2,0) -- (2,2) -- (0,2) -- (1,3.732) -- (2,2) --
(3.732,3) -- (5.464,2) -- (6.464,3.732) -- (4.732,4.732) --
(3.732,3) -- (2.732,4.732) -- (1,3.732) -- (0,2) (2.732,4.732) -- (4.732,4.732);
\end{scope}
}
\begin{scope}[xshift=2.732 cm, yshift=4.732cm]
\foreach \y in {1,2}
\foreach \x in {1, ..., 5}
{
\begin{scope}[xshift=5.464*\x cm, yshift=9.464*\y cm]
\draw (0,2) -- (0,0) -- (2,0) -- (2,2) -- (0,2) -- (1,3.732) -- (2,2) --
(3.732,3) -- (5.464,2) -- (6.464,3.732) -- (4.732,4.732) --
(3.732,3) -- (2.732,4.732) -- (1,3.732) -- (0,2) (2.732,4.732) -- (4.732,4.732);
\end{scope}
}
\end{scope}
\end{scope}
\end{tikzpicture}

}}
\subfloat[\scriptsize{(3.3.3.3.6)~1}]{
\makebox[.35\textwidth]{
\begin{tikzpicture}[yscale=.2*.63*.58*1.7, xscale=.2*.63*.61*1.7]
\begin{scope}[yscale = 1.732]
\clip (4,3) rectangle (40,20);
\draw[dashed] (4,3) rectangle (40,20);
\draw[green!73!black, line width = 1mm] (17,7) -- (21,7) -- (23,9) -- (21,11) --
(17,11) -- (15,9) -- cycle;

\foreach \Yshift in {1,2,3,4}
\foreach \Xshift in {1,2,3,4}
{
\foreach \x/\y in {4/2, -5/1, 0/0, 5/-1, -4/-2, 1/-3}
{
\begin{scope}[xshift = \x cm+14*\Xshift cm-2*\Yshift cm, yshift = \Yshift*6cm + \y cm]
\draw (0,0) -- (-2,0) -- (-4,0) (-2,0) -- (-3,-1) -- (-4,0) -- (-5,-1)
(-3,-1) -- (-5,-1) -- (-4,-2) -- (-3,-1) -- (-2,-2) -- (-4,-2) -- (-3,-3) --
(-2,-2) (-3,-3) -- (-1,-3) -- (-2,-2) -- (0,-2) -- (1,-3) -- (-1,-3) -- (0,-2);
\end{scope}
}
}
\end{scope}
\end{tikzpicture}
}}

\subfloat[\scriptsize{(3.3.4.3.4)~1}]{
\makebox[.35\textwidth]{
\begin{tikzpicture}[yscale=.61*.98*1.7, xscale=.618*1.7, rotate=180]
\draw[dashed] (2.73,2.18) rectangle (7.1,5.98);
\clip (2.73,2.18) rectangle (7.1,5.98);
\begin{scope}[scale=.2]
\draw[green!73!black, line width = 1mm] (19.1,20.8) 
-- ++ (1.732,1) -- ++ (1,1.732) 
-- ++ (0,2) 
-- ++ (-1,1.732) -- ++ (-1.732,1)
-- ++ (-1.732,-1) -- ++ (-1,-1.732)
-- ++ (0,-2) 
-- ++ (1,-1.732) -- cycle; 
%

\foreach \y in {1,...,5}
\foreach \x in {2, ..., 6}
{
\begin{scope}[xshift=5.464*\x cm, yshift=5.464*\y cm]
\draw (1.732,2.732) -- (0,1.732) -- (1,0) -- (2.732,1) -- (4.464,0) --
(5.464,1.732) -- (3.732,2.732) -- (5.464,3.732) -- (4.464,5.464) --
(2.732,4.464) -- (1,5.464) -- (0,3.732) -- (1.732,2.732)--
(2.732,4.464)--(3.732,2.732) -- (2.732,1) -- (1.732,2.732) -- (3.732,2.732)
(2.732,4.464) -- (2.732,6.464) (5.464,1.732) -- (5.464,3.732) (4.464,5.464) --
(6.464,5.464);
\end{scope}
}
\end{scope}
\end{tikzpicture}
}}
\subfloat[\scriptsize{(3.3.3.4.4)~1}]{
\makebox[.35\textwidth]{
\begin{tikzpicture}[yscale = .26*.63*.95*1.7, xscale=-1*.26*.63*1.7]
\draw[dashed] (2,14.92) rectangle (18.0,29.82);
\draw[green!73!black, line width = 1mm] (7,20.65) -- (11,20.65) -- (12,22.40) --
(12,24.40) -- (11,26.15) -- (7, 26.15) -- (6, 24.4) -- (6, 22.4) -- cycle;
\begin{scope}
\clip (2,14.92) rectangle (18.0,29.82);
\foreach \y in {1,2,3,4}
\foreach \x in {1, ..., 8}
{
\begin{scope}[xshift=2*\x cm, yshift=7.464*\y cm]
\draw (0,2) -- (0,0) -- (2,0) -- (2,2) -- (0,2) -- (1,3.732) -- (2,2) --
(3,3.732) -- (1,3.732);
\end{scope}
}
\begin{scope}[xshift=1 cm, yshift=3.732cm]
\foreach \y in {1,2,3}
\foreach \x in {0, ..., 8}
{
\begin{scope}[xshift=2*\x cm, yshift=7.464*\y cm]
\draw (0,2) -- (0,0) -- (2,0) -- (2,2) -- (0,2) -- (1,3.732) -- (2,2) -- (3,3.732) -- (1,3.732);
\end{scope}
}
\end{scope}
\end{scope}
\end{tikzpicture}
}}

\caption{The 8 non-regular Archimedean lattices, along with the configurations
used to prove upper bounds on their percolation thresholds.\label{fig0}}
\end{figure}

\begin{thm}
For every $p\in (0,1)$, the percolation threshold for the lattice $(4.8.8)$ is $3$.
\end{thm}
\begin{proof}
By applying Observation~\ref{obs1}, using the configuration in
Figure~\ref{fig0}, we get an upper bound of 3.  So we need only to prove a
matching lower bound.

We draw (4.8.8) so each 8-gon has its center at a lattice
point (and each lattice point is the center of an 8-gon).  We fix $p\in (0,1)$,
and initially infect each face independently with probability $p$.  Call this
set of initially infected faces \Memph{$\I$}.  We show that in the 3-bootstrap
model $\I$ percolates with positive probability.  Let \Memph{$\D_t$} denote
the set of faces with centers at $(x,y)$ such that $|x|\le t$ and $|y|\le t$.
So $\D_t$ contains $(2t+1)^2$ 8-gons and $(2t)^2$ squares.

Suppose that all faces in $\D_t$ are infected.  We want to prove a lower bound
on the probability that eventually all faces in $\D_{t+1}$ become infected.
Suppose that some infected face $f$ is adjacent to a face in the top row of
8-gons of $\D_t$.  Now the infection at $f$ will spread to every face in the
same row as $f$ that is adjacent to some face in $\D_t$.  This spread happens as
follows.  First, we infect the two 4-gons that are adjacent to $f$ and also each
have two neighbors in $\D_t$.  Now we infect each 8-gon $f'$ that is adjacent to
$f$ and also adjacent to a face $\D_t$ (since $f'$ is also adjacent to an
infected 4-gon).  By repeating this argument with $f'$ in place of $f$, we see
that the infection spreads along the row above the top row of $\D_t$ (to the
full width of $\D_t$, which is $2t+1$).  Since each face in the row just above
$\D_t$ is initially infected with probability $p$, the probability that none is
infected is $(1-p)^{2t+1}$.  The same argument applies to the row beneath the
bottom of $\D_t$ and to the columns to the left and right of $\D_t$.  So the
probability that at least one of these two rows and two columns has no infected
face is at most $4(1-p)^{2t+1}$.  If we infect all of both columns and both
rows, then we also infect each square with exactly one neighbor in $\D_t$ (these
are at the corners).  Finally, we infect each corner 8-gon, since it now has two
adjacent infected 8-gons and one adjacent infected square.  Thus, all of
$\D_{t+1}$ becomes infected.

For each $s\ge 0$, call $\D_{s+1}\setminus \D_s$ a \Emph{ring} around $\D_0$.
We partition the faces of each ring into top and bottom rows, left and right
columns, and four corners.  To infect the whole plane, it is enough to have all
faces in $\D_t$ infected (for some $t$) and for each $s\ge t$ to have at least
one infected face in each of its top and bottom rows and right and left columns.
The probability of having at least one ring without the necessary infected faces
is at most $\sum_{s\ge t}4(1-p)^{2s+1} = 4(1-p)^{2t+1}/(1-(1-p)^2)$.  For $t$
sufficiently large, this probability is less than 1.  The probability that every
8-gon in $\D_t$ is initially infected is $p^{(2t+1)^2}$; if the 8-gons are all
infected, then the squares immediately become infected.
Since each face is infected independently, the probability of infecting the 
whole plane is at least $p^{(2t+1)^2}(1-4(1-p)^{2t+1})/(1-(1-p)^2)$, which is
positive for $t$ sufficiently large.  So, in the 3-bootstrap model, with
positive probability, the whole plane becomes infected.  Now we use Kolmogorov's
0--1 Law to show that, in fact, the whole plane becomes infected with
probability 1.  To apply the 0--1 Law, we only need to note that the event that
the initial set $\I$ percolates is weakly translation-invariant.
\end{proof}

The proofs of the lower bounds for (3.6.3.6), (4.6.12), and (6.6.6) are similar.
The only noticeable difference is the shapes of the sets analogous to $\D_t$ and
the details of how $\D_t$ grows to $\D_{t+1}$ when we have at least one infected
face on each side of the ring $\D_{t+1}\setminus \D_t$.  In each case, the shape
of the set $\D_t$ is closer to a hexagon than a square, so the ring
$\D_{t+1}\setminus \D_t$ has six sides, rather than four; this is most obvious
for (6.6.6).  In Figure~\ref{fig1} we show examples of how one side of $\D_t$
grows to $\D_{t+1}$ for (3.6.3.6) and (4.6.12).  The faces marked with $\times$
are already infected, and the integers denote the order that new faces
become infected.  For (4.6.12), the faces labeled 0 become infected
immediately, since each has three infected neighbors.

\begin{figure}
\centering
\begin{tikzpicture}
\begin{scope}[xscale =.25*.9, yscale = .433*.9]
\draw[dotted] (0,0) rectangle (24,12);
\clip (0,0) rectangle (24,12);
\foreach \y in {0,4,8}
{
\begin{scope}[yshift = \y cm]
\draw (0,4) -- (24,4) (0,2) --  (24,2) (0,0) -- (24,0) 
(3,4)--(7,0) (7,4)--(11,0) (11,4)--(15,0) (15,4)--(19,0) (19,4)--(23,0)
(-1,4) -- (3,0) (23,4) -- (24,3) (1,4) -- (0,3)
(1,0)--(5,4) (5,0)--(9,4) (9,0)--(13,4) (13,0)--(17,4) (17,0)--(21,4)
(21,0) -- (25,4);
\end{scope}
}

\draw (10,9.65) node {\tiny{$4$}};
\draw (14,9.65) node {\tiny{$4$}};

\draw (8,9) node {\large{$3$}};
\draw (12,9) node {\large{$1$}};
\draw (16,9) node {\large{$3$}};

\draw (6,8.35) node  {\tiny{$4$}};
\draw (10,8.35) node {\tiny{$2$}};
\draw (14,8.35) node {\tiny{$2$}};
\draw (18,8.35) node {\tiny{$4$}};

\draw (8,7.65) node {\tiny{$\times$}};
\draw (12,7.65) node {\tiny{$\times$}};
\draw (16,7.65) node {\tiny{$\times$}};

\draw (6,7) node  {\LARGE{$\times$}};
\draw (10,7) node {\LARGE{$\times$}};
\draw (14,7) node {\LARGE{$\times$}};
\draw (18,7) node {\LARGE{$\times$}};

\draw (4,6.35) node {\tiny{$\times$}};
\draw (8,6.35) node {\tiny{$\times$}};
\draw (12,6.35) node {\tiny{$\times$}};
\draw (16,6.35) node {\tiny{$\times$}};
\draw (20,6.35) node {\tiny{$\times$}};

\draw (6,7) node  {\LARGE{$\times$}};
\draw (10,7) node {\LARGE{$\times$}};
\draw (14,7) node {\LARGE{$\times$}};
\draw (18,7) node {\LARGE{$\times$}};

\draw (6,5.65) node  {\tiny{$\times$}};
\draw (10,5.65) node {\tiny{$\times$}};
\draw (14,5.65) node {\tiny{$\times$}};
\draw (18,5.65) node {\tiny{$\times$}};

\draw (4,5) node {\LARGE{$\times$}};
\draw (8,5) node {\LARGE{$\times$}};
\draw (12,5) node {\LARGE{$\times$}};
\draw (16,5) node {\LARGE{$\times$}};
\draw (20,5) node {\LARGE{$\times$}};

\draw (6,4.35) node  {\tiny{$\times$}};
\draw (10,4.35) node {\tiny{$\times$}};
\draw (14,4.35) node {\tiny{$\times$}};
\draw (18,4.35) node {\tiny{$\times$}};

\draw (4,3.65) node {\tiny{$\times$}};
\draw (8,3.65) node {\tiny{$\times$}};
\draw (12,3.65) node {\tiny{$\times$}};
\draw (16,3.65) node {\tiny{$\times$}};
\draw (20,3.65) node {\tiny{$\times$}};

\draw (6,3) node  {\LARGE{$\times$}};
\draw (10,3) node {\LARGE{$\times$}};
\draw (14,3) node {\LARGE{$\times$}};
\draw (18,3) node {\LARGE{$\times$}};

\draw (8,2.35) node {\tiny{$\times$}};
\draw (12,2.35) node {\tiny{$\times$}};
\draw (16,2.35) node {\tiny{$\times$}};

\end{scope}

\begin{scope}[yshift=-.47in, xshift=1.5in,scale=.523]
\clip (4.0,2.25) rectangle (17.2,11.25);
\draw[dotted] (4.0,2.25) rectangle (17.2,11.25);
\foreach \y in {1,2,3}
\foreach \x in {1, ..., 7}
{
\begin{scope}[xshift=2.64*\x cm, yshift=4.5*\y cm]
\foreach \t in {15, 45, ..., 345}
\draw (\t:1.05) -- (\t+30:1.05);
%
\draw (1.01,.28) -- (1.63,.28);
\draw (1.01,-.28) -- (1.63,-.28);
\draw (.7475,-.7347) -- (1.050,-1.250);
\draw (.2625,-1.015) -- (.574, -1.53);
\draw (-.7475,-.7347) -- (-1.050,-1.250);
\draw (-.2625,-1.015) -- (-.574, -1.53);
\end{scope}
}

\begin{scope}[xshift=1.32 cm, yshift=2.25cm]
\foreach \y in {0,1,2}
\foreach \x in {1, ..., 6}
{
\begin{scope}[xshift=2.64*\x cm, yshift=4.5*\y cm]
\foreach \t in {15, 45, ..., 345}
\draw (\t:1.05) -- (\t+30:1.05);

\draw (1.01,.28) -- (1.63,.28);
\draw (1.01,-.28) -- (1.63,-.28);
\draw (.7475,-.7347) -- (1.050,-1.250);
\draw (.2625,-1.015) -- (.574, -1.53);
\draw (-.7475,-.7347) -- (-1.050,-1.250);
\draw (-.2625,-1.015) -- (-.574, -1.53);
\end{scope}
}
\foreach \x/\y in {1/1,2/1,3/1,4/1,5/1,1/2,2/2,3/2,4/2} 
\draw (2.64*\x cm+\y*1.325cm, \y*2.25cm) node {\huge{$\times$}};
\foreach \x/\y in {1/1,2/1,3/1,4/1,1/2,2/2,3/2} 
\draw (2.64*\x cm+\y*1.325cm+1.325cm, \y*2.25cm) node {\large{$\times$}};
\foreach \x in {1,...,8} 
\draw (1.32*\x cm+2.5*1.32cm, 3.375cm) 
node[rotate=30*(-1)^\x]
{\large{$\times$}};
\foreach \x in {1,...,4} 
\draw (2.64*\x cm+1.32cm+1.32cm, 3.00cm) node {\large{$\times$}};
\foreach \x in {1,...,3} 
\draw (2.64*\x cm+1.32cm+2.65cm, 3.75cm) node {\large{$\times$}};
\foreach \x/\num in {1/5,2/1,3/5} 
{
\draw (2.64*\x cm+1.32cm+2.65cm, 6.75cm) node {\LARGE{\num}};
\draw (2.64*\x cm+1.32cm+2.65cm, 5.25cm) node {\large{$0$}};
}
\foreach \x/\num in {1/6,2/4,3/2,4/2,5/4,6/6} 
\draw (1.32*\x cm+3.5*1.32cm, 5.625cm) node[rotate=30*(-1)^\x]
{\scriptsize{$\num$}};
\foreach \x/\num in {1/6,2/6} 
\draw (2.64*\x cm+3*1.32cm+1.32cm, 3*2.25cm) node {\scriptsize{$\num$}};
\foreach \x/\num in {1/3,2/3} 
\draw (2.64*\x cm+3*1.32cm+1.32cm, 3*2.25cm-.75cm) node {\small{$\num$}};

\end{scope}

\end{scope}

\end{tikzpicture}
\caption{The order in which faces in the next row become infected, when proving
that the lattice $(3.6.3.6)$ has threshold 2 and that the lattice $(4.6.12)$ has
threshold 3.\label{fig1}}
\end{figure}
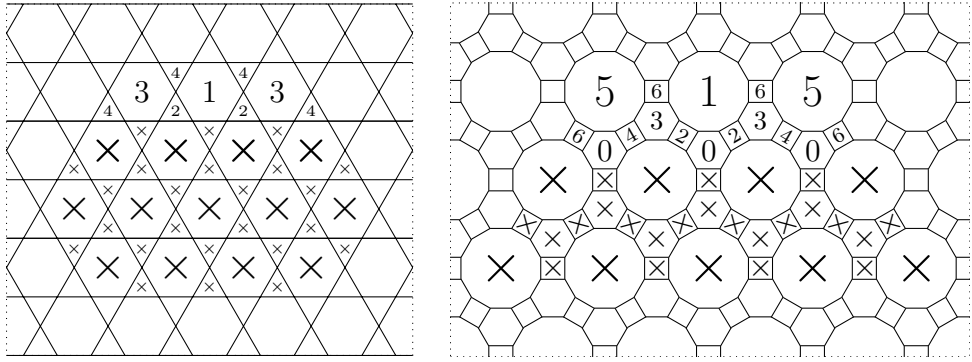

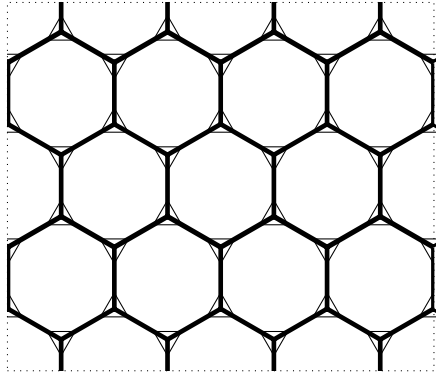
\begin{figure}[!h]
\centering
\begin{tikzpicture}[scale = .7]
\draw[dotted] (4.03,3.5) rectangle (12.13,10.5);
\begin{scope}
\clip (4.03,3.5) rectangle (12.15,10.5);
\foreach \y in {1,2,3}
\foreach \x in {1, ..., 6}
{
\begin{scope}[xshift=2.02*\x cm, yshift=3.51*\y cm]
\draw (45:1.05)  
\foreach \t in {30, 60, ..., 360}
{ -- (\t+15:1.05)} -- (45:1.05);
\draw[very thick] (30:1.15)  
\foreach \t in {15, 75, ..., 345}
{ -- (\t+15:1.15)} -- (30:1.15);

\end{scope}
}
\begin{scope}[xshift=1.01 cm, yshift=1.7505cm]
\foreach \y in {1,2}
\foreach \x in {1, ..., 5}
{
\begin{scope}[xshift=2.02*\x cm, yshift=3.51*\y cm]
\draw (45:1.05)  
\foreach \t in {30, 60, ..., 360}
{ -- (\t+15:1.05)} -- (45:1.05);
\draw[very thick] (30:1.15)  
\foreach \t in {15, 75, ..., 345}
{ -- (\t+15:1.15)} -- (30:1.15);

\end{scope}
}
\end{scope}
\end{scope}
\end{tikzpicture}
\caption{Inflating 12-gons to hexagons shows that (3.12.12) has threshold at
least 3, since (6.6.6) has threshold 3.\label{fig2}}
\end{figure}

The fact that (3.12.12) has bootstrap threshold at least 3 follows directly from
the fact that (6.6.6) does.  We inflate each 12-gon in (3.12.12) to include one
third of each incident triangle.  This produces (6.6.6), as shown in
Figure~\ref{fig2}.
When we inflate a 12-gon, it does not become incident to any new face.  Since
(6.6.6) has threshold 3, we conclude that in the 3-bootstrap model, with
probability 1 every 12-gon in (3.12.12) becomes infected.  And once the three
12-gons incident to a triangle become infected, so does the triangle.  Thus,
(3.12.12) has threshold at least 3.  Finally,
recall that the lattice (4.4.4.4) has threshold 2.  (This was proved by van Enter~\cite{vanEnter};  it is this proof which
inspired the present paper.)

\section{More General Tilings}
\label{sec3}
In a plane tiling by regular polygons, the \emph{vertex type} for a vertex $v$
is the cyclicly ordered list of the lengths of faces that meet at $v$.
Since the interior angle of a regular $t$-gon is known (its measure in degrees 
is $180(t-2)/t$), determining the set of all possible vertex types is a
simple exercise in diophantine equations. Up to reflection, we have 21 types.
These are 3.3.3.3.3.3, 3.3.3.3.6, 3.3.3.4.4, 3.3.4.3.4,
3.3.6.6, 3.6.3.6, 3.3.4.12, 3.4.3.12, 3.4.4.6, 3.4.6.4, 4.4.4.4, 3.7.42, 3.8.24,
3.9.18, 3.10.15, 3.12.12, 4.5.20, 4.6.12, 4.8.8, 5.5.10, 6.6.6.  (Analyzing
these 21 possibilities gives a straightforward, albeit tedious, proof that there
are only 11 Archimedean lattices.) Gr\"{u}nbaum and Shephard~\cite{GS} give
nice pictures of the 21 types, as well as many plane tilings by regular
polygons.

Let $\T$ denote the set of plane tilings such that if $T\in \T$ and some vertex
type appears in $T$, then that type appears in $T$ infinitely often.  
It is easy to see that $\T$ contains more tilings than just the Archimedean
Lattices.  A portion of such a tiling is shown in Figure~\ref{fig5}.
We prove the following.

\begin{mainthm}
No tiling in $\T$ has threshold $4$ or more, and the only tilings in $\T$ with
threshold $3$ are the lattices $(3.12.12)$, $(4.6.12)$, $(4.8.8)$, and $(6.6.6)$.
\end{mainthm}
\begin{proof}
Fix $T\in \T$.  As a warmup, we show that $T$ has
threshold at most 4.  Suppose $T$ has a vertex $v$ of type
other than 5.5.10 and 6.6.6.  Note, by examining the 21 types above, that $v$
has an incident 3-face or 4-face.  So, by definition, $T$ has infinitely many
3-faces or 4-faces.  Now Observation~\ref{obs1} shows that $T$ has threshold at
most 4.  As we show in the next paragraph, type 5.5.10 cannot appear in any
plane tiling.  Finally, if $T$ has only vertex type 6.6.6, then $T$ is the
lattice (6.6.6), which has threshold 3.

The rest of the proof simply refines the idea in the previous paragraph.  We
first show that six types cannot appear in $T$ at all.  Suppose that $T$
contains a vertex of type 3.7.42.  Since no other type contains 7-gons or
42-gons, the lengths of faces incident to this 3-gon must alternate between 7
and 42.  But this is impossible, since 3 is odd.  So $T$ contains no vertex of
type 3.7.42.  Similar arguments show that $T$ contains no vertex of any of types
3.8.24, 3.9.18, 3.10.15, 4.5.20, and 5.5.10.

For the remaining types $t$ other than 3.6.3.6, 3.12.12, 4.6.12, 4.8.8, and
6.6.6, we show that if $T$ contains type $t$, then $T$ contains a configuration
$H$ where each face of $H$ has at most 2 adjacent faces outside $H$.  Since
$H$ appears infinitely often, by Observation~\ref{obs1} the threshold of $T$
is at most 2, as desired.  The details follow.

If $T$ contains two adjacent triangles, then we take these as $H$.  This handles
six types, leaving only 3.4.4.6, 3.4.6.4, 3.4.3.12, and 4.4.4.4.  If $v$ has
type 4.4.4.4, then $H$ is its four incident squares.  If $v$ has type 3.4.3.12,
then $H$ is the two incident triangles and the incident square.  If $v$ has type
3.4.4.6, then a short analysis shows that $T$ contains one of the configurations
on the left in Figure~\ref{fig3}.  Finally, if $v$ has type 3.4.6.4, then a
(slightly longer) proof shows that $T$ contains the configuration on the right in
Figure~\ref{fig4} (or else contains two triangles linked by one or two squares,
similar to the cases on the left of Figure~\ref{fig4}).


\begin{figure}[!h]
\centering
\begin{tikzpicture}[scale=.3]
\fill (2,3.45) circle (.3cm); 
\fill (9,3.45) circle (.3cm); %
\fill (16,3.45) circle (.3cm);%

\begin{scope}[yscale=1.732]
\draw(3,3) --++ (-1,1) --++ (-2,0) --++ (-1,-1) --++ (1,-1)
--++ (0,-1.155) --++ (-1,-1) --++ (1,-1) --++ (2,0) --++ (1,1) --++ (1,1)
--++ (-2,0) --++ (-2,0) --++ (0,1.155) --++ (2,0) --++ (2,0);
\draw[ultra thick] (2,2) --++ (1,1) --++ (1,-1) --++ (0,-1.155) --++ (-1,-1)
--++ (-1,1) -- cycle;

\begin{scope}[xshift = 7cm]
\draw(4,2) --++ (-2,2) --++ (-2,0) --++ (-1,-1) --++ (1,-1)
--++ (0,-1.155) --++ (4,0) (3,3) --++ (-1,-1) --++ (0,-2.31);
\draw[ultra thick] (4,2) --++ (-4,0) --++ (0,-2.31) --++ (4,0) -- cycle;
\end{scope}

\begin{scope}[xshift = 14cm]
\draw(3,3) --++ (-1,1) --++ (-2,0) --++ (-1,-1) --++ (1,-1)
--++ (0,-1.155) --++ (2,-2) --++ (2,0) --++ (1,1) --++ (-1,1) (4,2) --++ (-2,0)
--++ (0,-1.155) --++ (-2,0);
\draw[ultra thick] (0,2) --++ (2,0) --++ (1,1) --++ (1,-1) --++ (0,-1.155) --++
(-2,0) --++ (-1,-1) --++ (-1,1) -- cycle;
\end{scope}

\end{scope}

\begin{scope} [scale = 3, xshift = 4in, yshift=.5in] 
\fill (-1.56,0.28) circle (.1cm); %
\foreach \t in {15, 75, ..., 345}
{
\draw (\t+60:1.05)
\foreach \s in {90, 180, 270, 360}
{ -- ++(\s+\t-135:.543)};
\draw (\t:1.05) 
\foreach \s in {60,120,..., 360}
{ -- ++(\s+\t-75:.543)};
\draw (\t:1.05) -- ++ (\t-15:.543) 
-- ++ (\t+45:.543)
-- ++ (\t-45:.543)
-- ++ (\t-135:.543)
-- ++ (\t-225:.543);
\draw (\t+30:1.05) -- ++ (\t+45:.543) 
-- ++ (\t+75:.543)
-- ++ (\t-15:.543)
-- ++ (\t-105:.543)
-- ++ (\t-195:.543);
}
\end{scope}

\end{tikzpicture}
\caption{Left: The three possibilities for $C$ when $T$ contains a vertex of type
3.4.4.6.\label{fig3}
Right: A configuration, $C$, of 31 faces in which each face has at most two
neighbors outside $C$.\label{fig4}}
\end{figure}
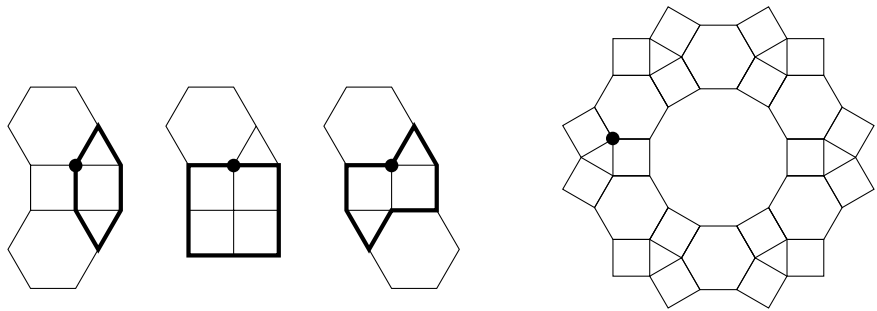

%

The remaining types to consider are 3.6.3.6, 3.12.12, 4.6.12, 4.8.8,
and 6.6.6.  To see that $T$ must be an Archimedean lattice, note that none of
these types agree in two or more successive face lengths.  So it is impossible
for $T$ to ``switch'' from one type to another.
\end{proof}

It is worth noting that we cannot relax the hypothesis in the Main Theorem to
require only that \emph{some} vertex type appears infinitely often.  For
example, suppose we start with the hex lattice and replace finitely many
hexagons each with 6 triangles.  If any of the resulting vertices of type
3.3.3.3.3.3 has no incident faces initially infected, then the percolation
threshold drops from 3 to 1.  Hence, the percolation threshold depends heavily
on $p$, the probability that each face is initially infected.

To conclude, we briefly discuss a family of tilings we call
$\mathcal{T}_{\textrm{strips}}$.  These tilings are formed by ``stacking''
infinite horizontal strips of polygons above and below each other to fill the
entire plane.  The two types of strips that we
use are \Emph{hex strips}, consisting of hexagons and triangles, and
\Emph{square strips}, consisting just of squares.  Figure~\ref{fig:strips}
shows an example.  
Since the hex strips can be
shifted left or right, this family contains uncountably many tilings.

Despite the variety in the tilings of
$\mathcal{T}_{\textrm{strips}}$, they all have the same threshold.
 The proof is similar to our proof for the lattice (4.8.8),
with a little difficulty added by the irregularly shaped rings we use now (what
were previously $\D_{t+1}\setminus \D_t$).  Two hex strips are
\Emph{offset} if the centers of their hexagons are not directly above one
another.

\begin{thm}\label{stripsthm}
Every tiling in $\mathcal{T}_{\textrm{strips}}$ has percolation threshold $2$.
\end{thm}

\begin{figure}[!t]
\centering
\begin{tikzpicture}[xscale=.25, yscale=.433]

\draw[dashed] (-2,-11.5) rectangle (46,11.2);
\clip (-2,-11.55) rectangle (46,11.2);
\foreach \y in {-1,-2,5.196,-8.2}
\foreach \x in {-2,...,24} 
{
\begin{scope}[xshift=\x*2cm+1cm, yscale=1.1547, yshift=\y*1cm]
\draw (0,0) -- (2,0) -- (2,1) -- (0,1) -- cycle;
\end{scope}
}

\foreach \y in {0,.5,-1.57735,1.79,2.29, -2.88} 
\foreach \x in {0,...,6} 
{
\begin{scope}[xshift=\x*8cm-4cm, yshift = \y*4cm]
\draw (0,1) -- (1,0) -- (3,0) -- (4,1) -- (5,0) -- (7,0) -- (8,1) -- (7,2) --
(5,2) -- (4,1) -- (3,2) -- (1,2) -- (0,1);
\draw (0,0) -- (8,0) (8,2) -- (0,2);
\end{scope}
}

\foreach \y in {1,-1.07735,-2.07735}
\foreach \x in {0,...,6} 
{
\begin{scope}[xshift=\x*8cm-2cm, yshift = \y*4cm]
\draw (0,1) -- (1,0) -- (3,0) -- (4,1) -- (5,0) -- (7,0) -- (8,1) -- (7,2) --
(5,2) -- (4,1) -- (3,2) -- (1,2) -- (0,1);
\draw (0,0) -- (8,0) (8,2) -- (0,2);
\end{scope}
}

\draw[ultra thick] (21,0) -- (23,0) -- (23,-1.1547) --++ (-2,0) -- cycle;

\draw[ultra thick] (5,0) --++ (-1,1) --++ (1,1) --++ (-1,1) --++ (3,3) --++
(0,1.1547) --++ (30,0) --++ (0,-1.1547) --++ (3,-3) --++ (-1,-1) --++ (1,-1)
--++ (-1,-1) --++ (0,-2.3094) --++ (-6,-6) --++ (-22,0) --++ (-6,6) -- cycle;

\end{tikzpicture}
\caption{A tiling in $\T_{\mbox{strips}}$, along with a marked face and
$\A_4$.\label{fig:strips}\label{fig5}}
\end{figure}
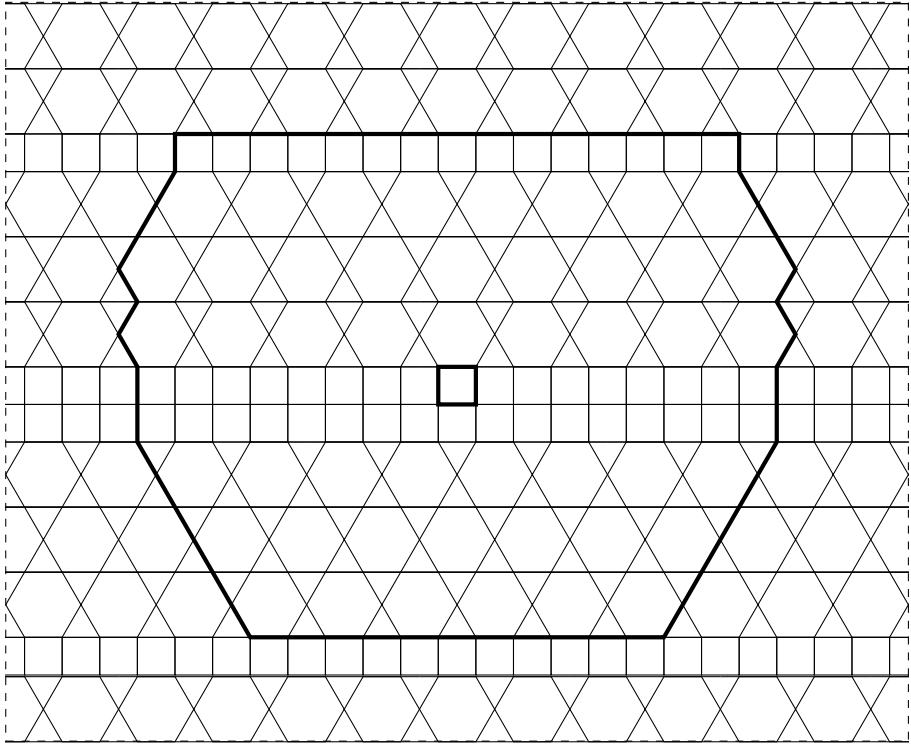
\begin{proof} 
Let $T$ be a tiling in $\mathcal{T}_{\textrm{strips}}$.  
Again the upper bound follows from Observation~\ref{obs1}.  The main step is to
show that $T$ contains a configuration $C$ such that each face of $C$ has
at most two adjacent faces outside $C$. A short analysis yields that $T$
contains infinitely many copies of one of the following: (a) adjacent triangles,
(b) a hexagon with six adjacent triangles, (c) four squares incident to a common
vertex, (d) two triangles adjacent to a common square, or (e) two triangles
linked by two squares (as in Figure~\ref{fig4}).

Now we show that for every $p$ with $0<p<1$, if $k=2$, then our random set $\I$
percolates with positive probability.  By combining this with the $0$--$1$ Law,
we conclude that the bootstrap threshold for $T$ is 2.

First we must find an analogue of $\D_t$ from our proof for the lattice (4.8.8).
Consider a face \Memph{$f$} of $T$ which is not a triangle. We let
\Memph{$\A_t$} denote a collection of faces that is centered on $f$ and that is
shaped somewhere between a square and a hexagon (depending on the number of
offset rows involved).  In the strip containing $f$, $\A_t$
contains $2t$ consecutive faces to the left of $f$ (including triangles),
and $2t$ consecutive faces to the right of $f$. For the strip above this,
$\A_t$ contains the faces directly above, if the two strips are not offset,
and the faces above and slightly towards the center, if the strips are offset. 
Similarly for the strip below, $\A_t$ contains the faces directly below if
the two strips are not offset, and the faces below and slightly toward the
center when the faces are offset.  We continue this for the $t$ rows above $f$
and the $t$ rows below $f$.  This means that $\A_t$ always consists of
$2t+1$ rows of faces, but the number of faces in the rows decreases slightly
as we move away from the center row (whenever successive strips are
offset).

Now $\A_t$ looks like a square when it has no offset strips, and
looks closer to a hexagon when it has many.  Even when $\A_t$ looks like
a rectangle, we think of $\A_{t+1}\setminus \A_t$ as having six sides. 
The top and bottom sides are easy to see; they consist of faces directly
above/below the faces in the top/bottom row of $\A_t$.  The top-left side 
consists of faces directly left of an end-face of $\A_t$ and which are in a
strip above $f$.  The bottom-left, top-right, and top-left sides are defined
similarly.

The key insight is that, just like for the lattice (4.8.8),
if $\A_t$ is infected and $\A_{t+1}\setminus
\A_t$ has even a single infected (non-triangular) face in one of its sides,
then that entire side becomes infected.
By repeatedly applying this idea, we see that the infection spreads along the
entire top-left side.  Once two adjacent sides are infected (e.g.~top-left and
top, or bottom-right and top-right), the corner face lying between them also 
has two infected neighbors, so it becomes infected.  The important consequence
of all this is the following.  If $\A_t$ is completely infected, and at
least one face on each of the six sides of $\A_{t+1}\setminus \A_t$ is
infected, then $\A_{t+1}$ also becomes completely infected.

Now we bound the probability that this happens.
Each side has at least
$t$ non-triangular faces\footnote{
When $\A_t$ contains two successive hex strips that are offset, the row further
from $f$ contributes to $\A_t$ two fewer faces than the row nearer $f$
(including one fewer hex face).  Thus,
the top and bottom sides can each have as many as $2t+1$ adjacent
non-triangular faces.  But, this only helps us, since a side with more faces is
\emph{more} likely to have an infected face.
}, so the probability
that none of the faces on a side are infected is at most $(1-p)^t$. 
Thus, the probability that at least one side of $\A_{t+1}\setminus \A_t$
has no infected face is no more than $6(1-p)^t$.

Now our argument exactly follows that for (4.8.8).
The probability that at least one ring around
$\A_t$ does not become infected is at most
$\sum_{j=0}^{\infty}6(1-p)^{t+j}=\frac{6(1-p)^t}{p}$, and for large enough $t$
we have $\frac{6(1-p)^t}{p}<1$.  Now the probability that $\A_t$ is
initially entirely infected and that every ring around $\A_t$
contains an infected face on each of the six sides is
$p^{|\A_t|}(1-\frac{6(1-p)^t}{p})>0$.  Since this event is
translation-invariant in the horizontal direction, it is weakly
translation-invariant.  So the $0$--$1$ Law tells us that $\I$
percolates with probability $1$.
\end{proof}

\section*{Acknowledgments}
We are indebted to Branko Gr\"{u}nbaum and Geoffrey Shephard for their article
\emph{Tilings by Regular Polygons}~\cite{GS}, which inspired much of this work.
Thanks also to one of our referees, whose feedback helped improve our presentation.

\bibliographystyle{abbrvplain}
\bibliography{percolation}
\end{document}